\documentclass[12pt, reqno]{amsart}
\usepackage{amsmath, amsthm, amscd, amsfonts, amssymb, graphicx, color}
\usepackage[bookmarksnumbered, colorlinks, plainpages]{hyperref}
\hypersetup{colorlinks=true,linkcolor=red, anchorcolor=green, citecolor=cyan, urlcolor=red, filecolor=magenta, pdftoolbar=true}

\newtheorem{theorem}{Theorem}[section]
\newtheorem{lemma}[theorem]{Lemma}
\newtheorem{prop}[theorem]{Proposition}
\newtheorem{cor}[theorem]{Corollary}
\theoremstyle{definition}

\theoremstyle{remark}
\newtheorem{remark}[theorem]{\bf{Remark}}
\numberwithin{equation}{section}
\begin{document}

\title [ Numerical radius inequalities of $2 \times 2$ operator matrices] 
       {Numerical radius inequalities of $2 \times 2$ operator matrices }

\author[P. Bhunia and K. Paul]{Pintu Bhunia and Kallol Paul}

\address{(Bhunia) Department of Mathematics, Jadavpur University, Kolkata 700032, West Bengal, India}
\email{pintubhunia5206@gmail.com}

\address{(Paul) Department of Mathematics, Jadavpur University, Kolkata 700032, West Bengal, India}
\email{kalloldada@gmail.com;kallol.paul@jadavpuruniversity.in}

\thanks{First author would like to thank UGC, Govt. of India for the financial support in the form of Senior Research Fellowship}
\thanks{}
\thanks{}



\maketitle

\begin{abstract}
Several upper and lower bounds for the numerical radius of $2 \times 2$ operator matrices are developed which  refine and generalize the earlier related bounds. In particular, we show that if $B,C$ are bounded linear operators on a complex Hilbert space, then
\begin{eqnarray*}
 && \frac{1}{2}\max \left \{ \|B\|, \|C\| \right \}+\frac{1}{4} \left | \|B+C^*\|-\|B-C^*\| \right |\\
 &&\leq w \left(\left[\begin{array}{cc}
		0 & B\\
		C& 0
	\end{array}\right]\right)\\ 
&&\leq \frac{1}{2} \max \left\{\|B\|,\|C\|\right \}+\frac{1}{2}\max \left \{r^{\frac{1}{2}}(|B||C^*|),r^{\frac{1}{2}}(|B^*||C|)\right\}, 
\end{eqnarray*}
where $w(.)$, $r(.)$ and $\|.\|$ are the numerical radius, spectral radius and  operator norm of a bounded linear operator, respectively. We also obtain equality conditions for the numerical radius of the operator matrix  $\left[\begin{array}{cc}
	0 & B\\
	C& 0
\end{array}\right]$. As application of results obtained, we  show that if $B,C$ are self-adjoint operators then,
$\max \Big \{\|B+C\|^2 , \|B-C\|^2  \Big\}\leq  \left \|B^2+C^2 \right \|+2w(|B||C|). $

\end{abstract}

\section{\textbf{Introduction}}

\noindent Let  $ \mathcal{H}$ be a complex Hilbert space with inner product  $\langle.,.\rangle$ and  let $ \mathcal{B}(\mathcal{H}) $ be the collection of all bounded linear operators on $ \mathcal{H}.$  As usual the norm induced by the inner product $\langle.,.\rangle$  is denoted by $ \Vert . \Vert.$  For  $A \in \mathcal{B}(\mathcal{H})$, let $\|A\|$ be the operator norm of $A,$ i.e., $ \Vert A \Vert = \sup_{\Vert x \Vert =1} \Vert Ax \Vert.$  For  $A \in \mathcal{B}(\mathcal{H})$, $A^*$ denotes the adjoint of $A$ and $|A|, |A^*|$  respectively denote the positive part of $A,A^*$, i.e., $|A|= (A^*A)^{\frac{1}{2}} , |A^*|= (AA^*)^{\frac{1}{2}}.$ The real part and the imaginary part of $A$ are denoted by $\Re(A)$ and $\Im(A)$ respectively so that  $\Re(A)=\frac{A+A^*}{2}$ and $\Im(A)=\frac{A-A^*}{2\rm i}$.
 The numerical range of $A$, denoted by $W(A),$ is defined as  $W(A)=\big\{ \langle Ax,x\rangle~~:~~x\in {\mathcal{H}}, \|x\|=1\big\}.$ 
It is well known that $\overline{W(A)}$ is a  compact subset of $\mathbb{C}$. The famous Toeplitz-Hausdorff theorem states that the numerical range is a convex set.
\noindent The numerical radius of $A$, denoted by $w(A)$, is defined as 
$w(A)= \sup_{\|x\|=1} |\langle Ax,x\rangle|.$
The numerical radius is a norm on $\mathcal{B}(\mathcal{H}) $ satisfying
\begin{eqnarray}\label{eqv}
\frac{1}{2}\|A\|\leq w(A)\leq \|A\|,
\end{eqnarray}
and so  the numerical radius norm is equivalent to the operator norm. The inequality  (\ref{eqv}) is sharp,  $w(A) = \|A\|$ if $A$ is normal and $ w(A)=\frac{\|A\|}{2} $ if $A^2=0.$ The spectral radius of $A$, denoted as $r(A),$ is defined as 
$r(A):= \sup_{\lambda \in \sigma(A)} |\lambda|,$
where $\sigma(A)$ is the spectrum of $A$. Since $\sigma(A)\subseteq \overline{W(A)}$, $r(A)\leq w(A)$.  For further basic properties on the  numerical range and the numerical radius of bounded linear operators, we refer to \cite{GR}. Various refinements of (\ref{eqv}) have been obtained recently, a few of them are in \cite{BBP2,BBP4,BBP5,BBP1,BPN}.

\noindent The direct sum of two copies of $\mathcal{H}$ is denoted by $\mathcal{H}\oplus \mathcal{H}.$ If $A,B,C,D\in \mathcal{B}(\mathcal{H})$, then the operator matrix $\left[\begin{array}{cc}
A & B\\
C& D
\end{array}\right]$ can be considered as an operator on $\mathcal{H}\oplus \mathcal{H},$ and  is defined by $\left[\begin{array}{cc}
A & B\\
C& D
\end{array}\right]x=\left ( \begin{array}{c}
Ax_1+Bx_2\\
Cx_1+Dx_2
\end{array}   \right),  \forall x=\left ( \begin{array}{c}
x_1\\
x_2
\end{array}   \right)\in \mathcal{H}\oplus \mathcal{H}.$

In this paper, we obtain several  upper and lower bounds for the numerical radius of $2 \times 2$ operator matrices. The bounds obtained here improve and generalize the earlier related bounds. We also  obtain equality conditions for the numerical radius of $\left[\begin{array}{cc}
0 & B\\
C& 0
\end{array}\right],$ where $`0$' denotes the zero operator on $\mathcal{H}.$ An application of some of our obtained bounds, we give norm inequalities for sums and differences of self-adjoint operators. 

\section{\textbf{Main results}}
We begin this section with the following well known lemmas. The first lemma can be found in \cite[Lemma 2.1]{HKS}.

\begin{lemma}\label{lemmma}
	Let $A,B,C,D\in \mathcal{B}(\mathcal{H})$. Then
\begin{enumerate}
	
	\item $w \left(\left[\begin{array}{cc}
		A & 0\\
		0& D
	\end{array}\right]\right)=\max\{  w(A), w(D)\}$.

	\item $w \left(\left[\begin{array}{cc}
	A & B\\
	B& A
	\end{array}\right]\right)=\max \{w(A+B),w(A-B)\}.$
	
	In particular, $w \left(\left[\begin{array}{cc}
		0 & B\\
		B& 0
	\end{array}\right]\right)=w(B).$
	\end{enumerate}

\end{lemma}

The second lemma can be proved easily.

\begin{lemma}
	Let $A,D\in \mathcal{B}(\mathcal{H})$. Then
		$$ \left \| \left[\begin{array}{cc}
	A & 0\\
	0& D
	\end{array}\right]\right\|=\left \| \left[\begin{array}{cc}
	0 & A\\
	D& 0
	\end{array}\right]\right\|  =\max\{  \|A\|, \|D\| \}.$$
\end{lemma}

The third lemma can be found in \cite[pp. 75-76]{halmos} which is a mixed Schwarz inequality.

\begin{lemma}\label{lem-1}
	Let $A\in \mathcal{B}(\mathcal{H})$. Then 
	\[|\langle Ax,x\rangle|\leq \langle |A|x,x\rangle^{1/2}~~\langle |A^*|x,x\rangle^{1/2},~~\forall~~x\in \mathcal{H}.\]
\end{lemma}

The fourth lemma involving positive operators can be found in \cite[Cor. 2]{K}.
\begin{lemma}\label{lem-positive1}
Let $A,B\in \mathcal{B}(\mathcal{H})$ be positive. Then
\[\|A+B\|\leq \max\{\|A\|, \|B\|  \}+\left\|A^{1/2}B^{1/2}\right\|.\]
\end{lemma}

Our first result can be stated as the following theorem.

\begin{theorem}\label{th-1}
Let $B,C\in \mathcal{B}(\mathcal{H})$. Then 
\begin{eqnarray*}
w \left(\left[\begin{array}{cc}
	0 & B\\
	C& 0
\end{array}\right]\right) &\leq& \frac{1}{2} \max \left\{\|B\|,\|C\|\right \}+\frac{1}{2}\max \left \{r^{\frac{1}{2}}(|B||C^*|),r^{\frac{1}{2}}(|B^*||C|)\right\}.
\end{eqnarray*}

\noindent \text{This inequality is sharp.}

\end{theorem}

\begin{proof}
	Let $x\in \mathcal{H}\oplus\mathcal{H}$ with $\|x\|=1$. Then from Lemma \ref{lem-1} we have that
	\begin{eqnarray*}
	\left | \left \langle \left[\begin{array}{cc}
		0 & B\\
		C& 0
	\end{array}\right]x,x    \right \rangle \right| &\leq& 	 \left \langle  \left |\left[\begin{array}{cc}
	0 & B\\
	C& 0
\end{array}\right]\right|x,x \right \rangle^{\frac{1}{2}} \left \langle  \left |\left[\begin{array}{cc}
0 & C^*\\
B^*& 0
\end{array}\right]\right|x,x \right \rangle^{\frac{1}{2}}\\
&\leq& 	\frac{1}{2} \left ( \left \langle  \left |\left[\begin{array}{cc}
	0 & B\\
	C& 0
\end{array}\right]\right|x,x \right \rangle + \left \langle  \left |\left[\begin{array}{cc}
	0 & C^*\\
	B^*& 0
\end{array}\right]\right|x,x \right \rangle \right)\\
&=& 	\frac{1}{2} \left \langle  \left (  \left |\left[\begin{array}{cc}
	0 & B\\
	C& 0
\end{array}\right]\right|+ \left |\left[\begin{array}{cc}
0 & C^*\\
B^*& 0
\end{array}\right]\right| \right) x,x \right \rangle \\
&=& 	\frac{1}{2} \left \langle  \left[\begin{array}{cc}
	|C|+|B^*| & 0\\
	0& |B|+|C^*|
\end{array}\right]  x,x \right \rangle \\
&\leq & \frac{1}{2} w \left (  \left[\begin{array}{cc}
	|C|+|B^*| & 0\\
	0& |B|+|C^*|
\end{array}\right]\right ) \\
&=& \frac{1}{2} \max \{ \||C|+|B^*|\|,  \||B|+|C^*|\|     \}.
	\end{eqnarray*}
By considering the supremum over all $\|x\|=1$, we get
\begin{eqnarray}\label{th-1eqn}
	w\left ( \left[\begin{array}{cc}
		0 & B\\
		C& 0
	\end{array}\right] \right) &\leq& \frac{1}{2} \max \{ \||C|+|B^*|\|,  \||B|+|C^*|\|     \}.
\end{eqnarray}
Now it follows from Lemma \ref{lem-positive1} that 
\[\||C|+|B^*|\|\leq \max\{ \|B\|, \|C\|\}+ \| |C|^{\frac{1}{2}} |B^*|^{\frac{1}{2}}\|\] and 
\[\||B|+|C^*|\|\leq \max\{ \|B\|, \|C\|\}+ \| |B|^{\frac{1}{2}} |C^*|^{\frac{1}{2}}\|.\]

Hence,  from (\ref{th-1eqn}) we get, 
\begin{eqnarray*}
w\left ( \left[\begin{array}{cc}
	0 & B\\
	C& 0
\end{array}\right] \right) &\leq& \frac{1}{2} \max \{\|B\|, \|C\|     \}+\frac{1}{2} \max\{\| |B|^{\frac{1}{2}} |C^*|^{\frac{1}{2}}\|,\| |C|^{\frac{1}{2}} |B^*|^{\frac{1}{2}}\| \}.
\end{eqnarray*}
If  $A,B\in \mathcal{B}(\mathcal{H})$ are positive, then   $r^{\frac{1}{2}}(AB)=\left\|A^{1/2}B^{1/2}\right\|$, (  see \cite[Lemma 2.5]{BP2}).
Therefore, 
\begin{eqnarray*}
	w\left ( \left[\begin{array}{cc}
		0 & B\\
		C& 0
	\end{array}\right] \right) &\leq& \frac{1}{2} \max \{\|B\|, \|C\|     \}+\frac{1}{2} \max \left \{ r^{\frac{1}{2}}( |B| |C^*|),r^{\frac{1}{2}}( |C| |B^*| ) \right\}.
\end{eqnarray*}
This is  the required inequality. To show that the inequality is sharp, we consider $C=0$ so that  $	w\left ( \left[\begin{array}{cc}
	0 & B\\
	0& 0
\end{array}\right] \right) \leq \frac{\|B\|}{2} $, which is actually equal.
\end{proof}
\begin{remark}

In particular, considering $B=C$ in Theorem \ref{th-1} and using Lemma \ref{lemmma}, we get the inequality ( see \cite[Th. 2.1]{BP2})
$$w(B)\leq \frac{1}{2}\|B\|+\frac{1}{2}{r^{\frac{1}{2}}( |B| |B^*|)}.$$
Thus Theorem \ref{th-1} generalizes  \cite[Th. 2.1]{BP2}.
\end{remark}

We next obtain a lower bound for the numerical radius of the operator matrix $\left[\begin{array}{cc}
	0 & B\\
	C& 0
\end{array}\right].$

\begin{theorem}\label{th-5}
	Let $B,C\in \mathcal{B}(\mathcal{H})$. Then
	\[w\left(\left[\begin{array}{cc}
	0 & B\\
	C& 0
	\end{array}\right]\right) \geq \frac{1}{2}\max \{ \|B\|, \|C\|\}+\frac{1}{4}| \|B+C^*\|-\|B-C^*\|  |. \]
\end{theorem}

\begin{proof}
	We note that for any bounded linear operator $T$,  $w(T) \geq \| \Re(T)\| $ and $ w(T) \geq \| \Im (T)\|.$ 
	So we have, $  w\left(\left[\begin{array}{cc}
	0 & B\\
	C& 0
	\end{array}\right]\right) \geq \left \|\frac{B+C^*}{2}\right \|$ and $  w\left(\left[\begin{array}{cc}
	0 & B\\
	C& 0
	\end{array}\right]\right)\geq \left \|\frac{B-C^*}{2\rm i}\right \|.$ Then 
	\begin{eqnarray*}
		w\left(\left[\begin{array}{cc}
			0 & B\\
			C& 0
		\end{array}\right]\right) &\geq& \frac{1}{2} \max \{  \left \|{B+C^*}\right \|,  \left \|{B-C^*}\right \|  \}\\
		&=& \frac{1}{4}(\left \|{B+C^*}\right \|+ \left \|{B-C^*}\right \|)+  \frac{1}{4}|\left \|{B+C^*}\right \|- \left \|{B-C^*}\right \||\\
		&\geq& \frac{1}{4}\left \|(B+C^*) \pm (B-C^*)\right \|+  \frac{1}{4}|\left \|{B+C^*}\right \|- \left \|{B-C^*}\right \||.
	\end{eqnarray*} 
	This implies that $$w\left(\left[\begin{array}{cc}
	0 & B\\
	C& 0
	\end{array}\right]\right)\geq \frac{1}{2}\max \{ \|B\|, \|C\|\}+\frac{1}{4}|~~ \|B+C^*\|-\|B-C^*\| ~~|.$$
	This completes the proof.
\end{proof}

\begin{remark}
	
	In particular, considering $B=C$ in Theorem \ref{th-5},  we get
	$$w(B)\geq \frac{\|B\|}{2}+\frac{1}{4} \left |\|B+B^*\|-\|B-B^*\| \right|.$$
	Clearly, this is an improvement of the first inequality in (\ref{eqv}), i.e., $w(B)\geq \frac{\|B\|}{2}$.
\end{remark}

Next, we need the following lemma, known as   Buzano's extension of Schwarz inequality ( see \cite{B}).

\begin{lemma}\label{lem-2}
	If $x,y,e\in \mathcal{H}$ with $\|e\|=1$, then \[|\langle x,e\rangle \langle e,y\rangle|\leq \frac{1}{2} \left(\|x\| \|y\|+|\langle x,y \rangle|\right).\]
\end{lemma}

Using  the above lemma we prove the following theorem.

\begin{theorem}\label{th-2}
If $B,C\in \mathcal{B}(\mathcal{H})$, then
\begin{eqnarray*}
w^2 \left(\left[\begin{array}{cc}
	0 & B\\
	C& 0
\end{array}\right]\right) &\leq& \frac{1}{4} \max \left\{\|  |B|^2+|C^*|^2\|, \|  |B^*|^2+|C|^2 \|  \right \}\\
&&  +\frac{1}{2}\max \left\{w(|B||C^*|), w(|C||B^*|) \right\}.
\end{eqnarray*}
This inequality is sharp.
\end{theorem}

\begin{proof}
	Let $x\in \mathcal{H}\oplus\mathcal{H}$ with $\|x\|=1$. Then, 
	\begin{eqnarray*}
		\left | \left \langle \left[\begin{array}{cc}
			0 & B\\
			C& 0
		\end{array}\right]x,x    \right \rangle \right|^2 &\leq& 	 \left \langle  \left |\left[\begin{array}{cc}
			0 & B\\
			C& 0
		\end{array}\right]\right|x,x \right \rangle\left \langle  \left |\left[\begin{array}{cc}
			0 & C^*\\
			B^*& 0
		\end{array}\right]\right|x,x \right \rangle ,~~\text{by Lemma \ref{lem-1}}\\
	&=& 	 \left \langle  \left |\left[\begin{array}{cc}
		0 & B\\
		C& 0
	\end{array}\right]\right|x,x \right \rangle\left \langle   x,\left|\left[\begin{array}{cc}
	0 & C^*\\
	B^*& 0
\end{array}\right]\right|x \right \rangle \\
&=& 	 \left \langle  \left[\begin{array}{cc}
	|C| & 0\\
	0& |B|
\end{array}\right]x,x \right \rangle\left \langle   x,\left[\begin{array}{cc}
	|B^*| & 0\\
0	& |C^*|
\end{array}\right] x \right \rangle \\
&\leq & \frac{1}{2} \left \|  \left[\begin{array}{cc}
	|C| & 0\\
	0& |B|
\end{array}\right]x  \right\|  \left \| \left[\begin{array}{cc}
|B^*| & 0\\
0	& |C^*|
\end{array}\right] x\right \|\\
&& + \frac{1}{2} \left \langle\left[\begin{array}{cc}
	|C| & 0\\
	0& |B|
\end{array}\right]x ,  \left[\begin{array}{cc}
	|B^*| & 0\\
	0	& |C^*|
\end{array}\right] x\right \rangle, ~~\text{by Lemma \ref{lem-2}}\\
&=& \frac{1}{2} \sqrt{ \left \langle \left[\begin{array}{cc}
	|C|^2 & 0\\
	0& |B|^2
\end{array}\right]x,x  \right \rangle  \left \langle \left[\begin{array}{cc}
	|B^*|^2 & 0\\
	0	& |C^*|^2
\end{array}\right] x,x\right \rangle }\\
&& + \frac{1}{2}  \left \langle  \left[\begin{array}{cc}
	|B^*| & 0\\
	0	& |C^*|
\end{array}\right]\left[\begin{array}{cc}
	|C| & 0\\
	0& |B|
\end{array}\right]x ,  x\right \rangle\\
&\leq& \frac{1}{4} \left[ \left \langle \left[\begin{array}{cc}
		|C|^2 & 0\\
		0& |B|^2
	\end{array}\right]x,x  \right \rangle + \left \langle \left[\begin{array}{cc}
		|B^*|^2 & 0\\
		0	& |C^*|^2
	\end{array}\right] x,x\right \rangle \right]\\
&& + \frac{1}{2}  \left \langle  \left[\begin{array}{cc}
	|B^*||C| & 0\\
	0	& |C^*||B|
\end{array}\right]x ,  x\right \rangle\\
&=& \frac{1}{4}  \left \langle \left[\begin{array}{cc}
	|C|^2 +|B^*|^2& 0\\
	0& |B|^2+|C^*|^2
\end{array}\right]x,x  \right \rangle\\
 && + \frac{1}{2}  \left \langle  \left[\begin{array}{cc}
	|B^*||C| & 0\\
	0	& |C^*||B|
\end{array}\right]x ,  x\right \rangle\\
&\leq & \frac{1}{4}  w\left ( \left[\begin{array}{cc}
	|C|^2 +|B^*|^2& 0\\
	0& |B|^2+|C^*|^2
\end{array}\right]  \right )\\
&& + \frac{1}{2} w \left (  \left[\begin{array}{cc}
	|B^*||C| & 0\\
	0	& |C^*||B|
\end{array}\right] \right )\\
&=&\frac{1}{4} \max \{\| 	|C|^2 +|B^*|^2\|, \||B|^2+|C^*|^2 \|  \}\\
&&  + \frac{1}{2} \max \{w(|B^*||C|),w( |C^*||B|)\}.
\end{eqnarray*}
Taking supremum over all $ x \in \mathcal{H},\|x\|=1$, we get
\begin{eqnarray*}
	w^2 \left(\left[\begin{array}{cc}
		0 & B\\
		C& 0
	\end{array}\right]\right) &\leq& \frac{1}{4} \max \left\{\|  |B|^2+|C^*|^2\|, \|  |B^*|^2+|C|^2 \|  \right \}\\
	&&  +\frac{1}{2}\max \left\{w(|B||C^*|), w(|C||B^*|) \right\}.
\end{eqnarray*}
To show that the  inequality is sharp, we consider $C=0$. Then we get, $	w^2 \left(\left[\begin{array}{cc}
0 & B\\
0& 0
\end{array}\right]\right) \leq \frac{1}{4}\|  B\|^2,$ which is actually equal.
This completes the proof.
\end{proof}

\begin{remark}
In particular, considering $B=C$ in Theorem \ref{th-2} and using Lemma \ref{lemmma}, we get the inequality \cite[Th. 2.5]{P14}
	\[w^2(B)\leq \frac{1}{4}\| |B|^2+|B^*|^2\|+\frac{1}{2}w(|B||B^*|).\] 
	Thus Theorem \ref{th-2} generalizes  \cite[Th. 2.5]{P14}.
	\end{remark}
Our next result reads as follows.

\begin{theorem}\label{th-6}
	Let $B,C\in \mathcal{B}(\mathcal{H})$. Then
	\begin{eqnarray*}
		w^2\left(\left[\begin{array}{cc}
			0 & B\\
			C& 0
		\end{array}\right]\right) &\geq& \frac{1}{4}\max \left \{\| ~~ |B|^2+|C^*|^2~~\|, \|  ~~|B^*|^2+|C|^2 ~~\| \right \}\\
		&& +\frac{1}{8}|~~ \|B+C^*\|^2-\|B-C^*\|^2~~  |. 
	\end{eqnarray*}
	
\end{theorem}

\begin{proof}
	
	We note that for any bounded linear operator $T$,  $w(T) \geq \| \Re(T)\| $ and $ w(T) \geq \| \Im (T)\|.$ 
	So we have,  $  w\left(\left[\begin{array}{cc}
	0 & B\\
	C& 0
	\end{array}\right]\right) \geq \left \|\frac{B+C^*}{2}\right \|$ and $  w\left(\left[\begin{array}{cc}
	0 & B\\
	C& 0
	\end{array}\right]\right)\geq \left \|\frac{B-C^*}{2\rm i}\right \|.$  Then,
	\begin{eqnarray*}
		w^2\left(\left[\begin{array}{cc}
			0 & B\\
			C& 0
		\end{array}\right]\right) &\geq& \frac{1}{4} \max \{  \left \|{B+C^*}\right \|^2,  \left \|{B-C^*}\right \|^2  \}\\
		&=& \frac{1}{8}(\left \|{B+C^*}\right \|^2+ \left \|{B-C^*}\right \|^2)+  \frac{1}{8}|\left \|{B+C^*}\right \|^2- \left \|{B-C^*}\right \|^2|\\
		&=& \frac{1}{2} \left ( \left \|\frac{B+C^*}{2}\right \|^2+  \left \|\frac{B-C^*}{2\rm i}\right \|^2 \right )\\
		&& +  \frac{1}{8}|\left \|{B+C^*}\right \|^2- \left \|{B-C^*}\right \|^2|\\
		&=& \frac{1}{2}\left ( \left \| \Re \left (\left[\begin{array}{cc}
			0 & B\\
			C& 0
		\end{array}\right] \right)\right \|^2+ \left \| \Im \left (\left[\begin{array}{cc}
			0 & B\\
			C& 0
		\end{array}\right] \right)\right \|^2 \right)\\
		&& +  \frac{1}{8}|\left \|{B+C^*}\right \|^2- \left \|{B-C^*}\right \|^2|\\
		&=& \frac{1}{2}\left ( \left \| \Re^2 \left (\left[\begin{array}{cc}
			0 & B\\
			C& 0
		\end{array}\right] \right)\right \|+ \left \| \Im^2 \left (\left[\begin{array}{cc}
			0 & B\\
			C& 0
		\end{array}\right] \right)\right \| \right)\\
		&& +  \frac{1}{8}|\left \|{B+C^*}\right \|^2- \left \|{B-C^*}\right \|^2|\\
		&\geq & \frac{1}{2}\left ( \left \| \Re^2 \left (\left[\begin{array}{cc}
			0 & B\\
			C& 0
		\end{array}\right] \right)+ \Im^2 \left (\left[\begin{array}{cc}
			0 & B\\
			C& 0
		\end{array}\right] \right)\right \| \right)\\
		&& +  \frac{1}{8}|\left \|{B+C^*}\right \|^2- \left \|{B-C^*}\right \|^2|\\
		&= & 	\frac{1}{4} \left \| \left[\begin{array}{cc}
			|C|^2 +|B^*|^2& 0\\
			0& |B|^2+|C^*|^2
		\end{array}\right]  \right \|\\
		&& +  \frac{1}{8}|\left \|{B+C^*}\right \|^2- \left \|{B-C^*}\right \|^2|\\
		&=& \frac{1}{4}\max \left \{\| ~~ |B|^2+|C^*|^2~~\|, \|  ~~|B^*|^2+|C|^2 ~~\| \right \}\\
		&& +\frac{1}{8}|~~ \|B+C^*\|^2-\|B-C^*\|^2~~  |. 
	\end{eqnarray*} 
	This completes the proof.
\end{proof}

The following necessary condition for the equality of $w \left(\left[\begin{array}{cc}
0 & B\\
C& 0
\end{array}\right]\right) $ follows from Theorem \ref{th-6}.

\begin{prop}
	If $B,C \in \mathcal{B}(\mathcal{H})$, then $$w^2 \left(\left[\begin{array}{cc}
	0 & B\\
	C& 0
	\end{array}\right]\right) = \frac{1}{4} \max \left\{\|  |B|^2+|C^*|^2\|, \|  |B^*|^2+|C|^2 \|  \right \}$$ implies that  $~~\|B+C^*\|=\|B-C^*\|.$
\end{prop}

\begin{remark}
	In \cite[Th. 2.2]{BBP3}, the authors obtained that 
	\begin{eqnarray*}
		w^2\left(\left[\begin{array}{cc}
			0 & B\\
			C& 0
		\end{array}\right]\right) &\geq& \frac{1}{4}\max \left \{\| ~~ |B|^2+|C^*|^2~~\|, \|  ~~|B^*|^2+|C|^2 ~~\| \right \}. 
	\end{eqnarray*}
	Clearly, Theorem \ref{th-6} refines \cite[Th. 2.2]{BBP3}.
	
\end{remark}

Our next improvement of  \cite[Th. 2.2]{BBP3} is as follows.

\begin{theorem}\label{th-3}
If $B,C\in \mathcal{B}(\mathcal{H})$, then
\begin{eqnarray*}
	w^2 \left(\left[\begin{array}{cc}
		0 & B\\
		C& 0
	\end{array}\right]\right) &\geq& \frac{1}{8}\left[ \max \left\{ \|B+C^*\|^2, \|B-C^*\|^2    \right \}+ \|B+C^*\| \|B-C^*\|\right] \\
&\geq & \frac{1}{4} \max \left\{\|  |B|^2+|C^*|^2\|, \|  |B^*|^2+|C|^2 \|  \right \}.
\end{eqnarray*}	
The inequalities are sharp.
\end{theorem}

\begin{proof}
Let $\mathbb{S}=\left[\begin{array}{cc}
0 & B\\
C& 0
\end{array}\right]$. First inequality follows from $w(\mathbb{S}) \geq \|\Re(\mathbb{S})\| $ and $w(\mathbb{S}) \geq \|\Im(\mathbb{S})\|.$  We next prove the second inequality. Clearly, 
\[\frac{1}{4} \| | \mathbb{S}|^2+|\mathbb{S}^*|^2 \| =\frac{1}{2} \| \Re^2(\mathbb{S})+\Im^2(\mathbb{S})   \|.\]
Now, from Lemma \ref{lem-positive1}, we get
\begin{eqnarray*}
\| \Re^2(\mathbb{S})+\Im^2(\mathbb{S})   \| &\leq& \max \{ \|\Re^2(\mathbb{S})\|,  \|\Im^2(\mathbb{S})   \|  \} +  \| |\Re(\mathbb{S})| |\Im(\mathbb{S})|   \|\\
&=&   \max \{ \|\Re(\mathbb{S})\|^2,  \|\Im(\mathbb{S})   \|^2  \} +  \| |\Re(\mathbb{S})| |\Im(\mathbb{S})|   \|.
\end{eqnarray*}
Hence, we have
\begin{eqnarray*}
\frac{1}{4} \| | \mathbb{S}|^2+|\mathbb{S}^*|^2 \| &\leq& \frac{1}{2}\max \{ \|\Re(\mathbb{S})\|^2,  \|\Im(\mathbb{S})   \|^2  \} +   \frac{1}{2}\| |\Re(\mathbb{S})| |\Im(\mathbb{S})|   \|\\
&\leq& \frac{1}{2}\max \{ \|\Re(\mathbb{S})\|^2,  \|\Im(\mathbb{S})   \|^2  \} +   \frac{1}{2}\| |\Re(\mathbb{S})| \|   \| |\Im(\mathbb{S})|   \|\\
&=& \frac{1}{2}\max \{ \|\Re(\mathbb{S})\|^2,  \|\Im(\mathbb{S})   \|^2  \} +   \frac{1}{2}\| \Re(\mathbb{S}) \|   \| \Im(\mathbb{S})   \|. 
\end{eqnarray*}
This implies that
\begin{eqnarray*}
\frac{1}{4} \left \| \left[\begin{array}{cc}
	|C|^2 +|B^*|^2& 0\\
	0& |B|^2+|C^*|^2
\end{array}\right]  \right \| &\leq& \frac{1}{2} \max \left \{ \left \|\frac{B+C^*}{2}\right \|^2,\left \|\frac{B-C^*}{2\rm i}\right \|^2 \right \}\\
&& + \frac{1}{2} \left \|\frac{B+C^*}{2}\right \| \left \|\frac{B-C^*}{2 \rm i}\right \|,
\end{eqnarray*} 
that is,
\begin{eqnarray*}
\frac{1}{4} \max \left\{\|  |B|^2+|C^*|^2\|, \|  |B^*|^2+|C|^2 \|  \right \} &\leq&  \frac{1}{8} \max \left\{ \|B+C^*\|^2, \|B-C^*\|^2    \right \}\\
&& +  \frac{1}{8} \|B+C^*\| \|B-C^*\|. 
\end{eqnarray*}
This is the second inequality of the theorem. To show that the inequalities are sharp, we consider $C=0$. Then we get $	w^2 \left(\left[\begin{array}{cc}
0 & B\\
0& 0
\end{array}\right]\right) \geq \frac{1}{4}\|  B\|^2,$ which is actually equal.
This completes the proof.
\end{proof}

The following sufficient condition for the equality of $w \left(\left[\begin{array}{cc}
0 & B\\
C& 0
\end{array}\right]\right) $ follows from  Theorem \ref{th-2} and Theorem \ref{th-3}.

\begin{prop}
	Let $B,C \in \mathcal{B}(\mathcal{H})$. If $|B||C^*|=|B^*||C|=0$, then $$w^2 \left(\left[\begin{array}{cc}
	0 & B\\
	C& 0
	\end{array}\right]\right) = \frac{1}{4} \max \left\{\|  |B|^2+|C^*|^2\|, \|  |B^*|^2+|C|^2 \|  \right \}.$$
\end{prop}

\begin{remark}
	
In particular,  considering $B=C$ in Theorem \ref{th-3} and using $w\left ( \left[\begin{array}{cc}
	0 & B\\
	B& 0
	\end{array}\right] \right) =w\left (B \right) $, we get
	\begin{eqnarray*}
	w^2(B)&\geq& \frac{1}{8}\left[ \max \left\{ \|B+B^*\|^2, \|B-B^*\|^2    \right \}+ \|B+B^*\| \|B-B^*\|\right] \\
	&\geq&  \frac{1}{4}\| |B|^2+|B^*|^2\|.
	\end{eqnarray*}  
	 Thus Theorem \ref{th-3} generalizes  \cite[Th. 2.10]{BP1}.
	
\end{remark}

For next result we need the following lemma ( see  \cite[Th. 2.4]{BBP3}).

\begin{lemma}\label{lem-4}
If $A,B \in \mathcal{B}(\mathcal{H})$, then
\[\|A+B\|^2\leq 2 \max  \left \{\| |A|^2+|B|^2\|, \|  |A^*|^2+|B^*|^2\|  \right \}.\]
\end{lemma}


\begin{theorem}\label{th-4}
	If $B,C\in \mathcal{B}(\mathcal{H})$, then
	\begin{eqnarray*}
		w^2 \left(\left[\begin{array}{cc}
			0 & B\\
			C& 0
		\end{array}\right]\right) &\geq& \frac{1}{4\sqrt{2}}\left[  \|B+C^*\|^4+ \|B-C^*\|^4  \right]^{\frac{1}{2}} \\
		&\geq & \frac{1}{4} \max \left\{\|  |B|^2+|C^*|^2\|, \|  |B^*|^2+|C|^2 \|  \right \}.
	\end{eqnarray*}	
	The inequalities are sharp.
\end{theorem}
\begin{proof}
	Let $\mathbb{S}=\left[\begin{array}{cc}
	0 & B\\
	C& 0
	\end{array}\right]$. 
First inequality follows from $w(\mathbb{S}) \geq \|\Re(\mathbb{S})\| $ and $w(\mathbb{S}) \geq \|\Im(\mathbb{S})\|.$  We next prove the second inequality.  Clearly, 
	\[\frac{1}{4} \| | \mathbb{S}|^2+|\mathbb{S}^*|^2 \| =\frac{1}{2} \| \Re^2(\mathbb{S})+\Im^2(\mathbb{S})   \|.\]
	Now, from Lemma \ref{lem-4}, we have
	\begin{eqnarray*}
		\| \Re^2(\mathbb{S})+\Im^2(\mathbb{S})   \| &\leq& \sqrt{2} \|\Re^4(\mathbb{S})+\Im^4(\mathbb{S})   \|^{\frac{1}{2}}  \\
		&\leq &  \sqrt{2} \left[ \|\Re(\mathbb{S})\|^4+  \|\Im(\mathbb{S})   \|^4 \right]^{\frac{1}{2}}.
	\end{eqnarray*}
	Hence, we have
	\begin{eqnarray*}
		\frac{1}{4} \| | \mathbb{S}|^2+|\mathbb{S}^*|^2 \| &\leq& \frac{1}{\sqrt{2}} \left[ \|\Re(\mathbb{S})\|^4+  \|\Im(\mathbb{S})   \|^4 \right]^{\frac{1}{2}}.
	\end{eqnarray*}
	This implies that
	\begin{eqnarray*}
		\frac{1}{4} \left \| \left[\begin{array}{cc}
			|C|^2 +|B^*|^2& 0\\
			0& |B|^2+|C^*|^2
		\end{array}\right]  \right \| &\leq& \frac{1}{\sqrt{2}} \left [ \left \|\frac{B+C^*}{2}\right \|^4+\left \|\frac{B-C^*}{2\rm i}\right \|^4\right ]^{\frac{1}{2}},
	\end{eqnarray*} 
	that is,
	\begin{eqnarray*}
		\frac{1}{4} \max \left\{\|  |B|^2+|C^*|^2\|, \|  |B^*|^2+|C|^2 \|  \right \} &\leq&  \frac{1}{4\sqrt{2}}  \left[ \|B+C^*\|^4+ \|B-C^*\|^4    \right ]^{\frac{1}{2}}.
	\end{eqnarray*}
This is the second inequality of the theorem. To show that the inequalities are sharp, we consider $C=0$. Then we get $	w^2 \left(\left[\begin{array}{cc}
0 & B\\
0& 0
\end{array}\right]\right) \geq \frac{1}{4}\|  B\|^2,$ which is actually equal.
This completes the proof.

\end{proof}

\begin{remark}
In particular,	considering $B=C$ in Theorem \ref{th-4}, we get the inequality ( see \cite[Th. 2.13]{BP1})
	\[w^2(B)\geq  \frac{1}{4\sqrt{2}}\left[  \|B+B^*\|^4+ \|B-B^*\|^4  \right]^{\frac{1}{2}} \geq  \frac{1}{4}\|  |B|^2+|B^*|^2\|\]
	and so Theorem \ref{th-4} is a generalization of  \cite[Th. 2.13]{BP1}.
\end{remark}

For our next result we need the following lemmas.

\begin{lemma}$($\cite[p. 20]{simon}$)$.\label{positive}
	Let $A\in \mathcal{B}(\mathcal{H})$ be positive, i.e., $A\geq 0.$ Then \[\langle Ax,x\rangle^r\leq \langle A^rx,x\rangle,\] for all $r\geq 1$ and for all $x \in \mathbb{H}$ with $\|x\|=1.$
\end{lemma}

\begin{lemma}\label{lem-5}
Let $x,y,e\in \mathcal{H}$ with $\|e\|=1$. Then we have, for $0\leq \alpha \leq 1$
\begin{eqnarray*}
|\langle x,e\rangle\langle e,y\rangle|^2 &\leq& \frac{1+\alpha}{4}\|x\|^2\|y\|^2+\frac{1-\alpha}{4}|\langle x,y\rangle|^2+\frac{1}{2}\|x\|\|y\||\langle x,y\rangle|.
\end{eqnarray*}
\end{lemma}
  
  \begin{proof}
  From Lemma \ref{lem-2}, we have
  \begin{eqnarray*}
  |\langle x,e\rangle\langle e,y\rangle|^2 &\leq&  \frac{1}{4}(\|x\|\|y\|+|\langle x,y\rangle|)^2\\
  &=&\frac{1}{4} (\|x\|^2\|y\|^2+2\|x\|\|y\||\langle x,y\rangle|+|\langle x,y\rangle|^2)\\
   &=&\frac{1}{4} (\|x\|^2\|y\|^2+2\|x\|\|y\||\langle x,y\rangle|+\alpha |\langle x,y\rangle|^2+(1-\alpha )|\langle x,y\rangle|^2)\\
   &\leq &\frac{1}{4} (\|x\|^2\|y\|^2+2\|x\|\|y\||\langle x,y\rangle|+\alpha \|x\|^2\|y\|^2+(1-\alpha )|\langle x,y\rangle|^2)\\
   &\leq& \frac{1+\alpha}{4}\|x\|^2\|y\|^2+\frac{1-\alpha}{4}|\langle x,y\rangle|^2+\frac{1}{2}\|x\|\|y\||\langle x,y\rangle|,
  \end{eqnarray*}
as desired.
  \end{proof}

Now, we are in a position to prove our next result.

\begin{theorem}\label{th-7}
	If $B,C\in \mathcal{B}(\mathcal{H})$, then for $0\leq \alpha \leq 1$, we have
	\begin{eqnarray*}
 	&& w^4 \left(\left[\begin{array}{cc}
		0 & B\\
		C& 0
	\end{array}\right]\right) \\
 &\leq&\frac{1+\alpha}{8} \max \left\{\|  |B|^4+|C^*|^4\|, \|  |B^*|^4+|C|^4 \|  \right \}\\ &&+\frac{1-\alpha}{4} \max \left\{ w^2(BC),w^2(CB) \right\}\\
&&+ \frac{1}{4} \max \left\{\|  |B|^2+|C^*|^2\|, \|  |B^*|^2+|C|^2 \|  \right \}  \times \max \left\{ w(BC),w(CB) \right\}.
\end{eqnarray*}		
	
\end{theorem}

\begin{proof}
Let $\mathbb{S}=\left[\begin{array}{cc}
0 & B\\
C& 0
\end{array}\right].$  
Let  $x\in \mathcal{H}\oplus \mathcal{H}$ with $\|x\|=1$. Then it follows from Lemma \ref{lem-5} that
\begin{eqnarray*}
|\langle \mathbb{S}x,x\rangle|^4
&=&|\langle \mathbb{S}x,x\rangle\langle x,\mathbb{S}^*x\rangle|^2\\
&\leq &\frac{1+\alpha}{4}\|\mathbb{S}x\|^2\|\mathbb{S}^*x\|^2+\frac{1-\alpha}{4}|\langle \mathbb{S}^2x,x\rangle|^2+\frac{1}{2}\|\mathbb{S}x\|\|\mathbb{S}^*x\||\langle \mathbb{S}^2x,x\rangle|\\
&\leq & \frac{1+\alpha}{8}(\|\mathbb{S}x\|^4+\|\mathbb{S}^*x\|^4)+\frac{1-\alpha}{4}|\langle \mathbb{S}^2x,x\rangle|^2\\
&&+\frac{1}{4}(\|\mathbb{S}x\|^2+\|\mathbb{S}^*x\|^2)|\langle \mathbb{S}^2x,x\rangle|\\
&\leq &\frac{1+\alpha}{8}\langle (|\mathbb{S}|^4 +|\mathbb{S}^*|^4)x,x\rangle+\frac{1-\alpha}{4}|\langle \mathbb{S}^2x,x\rangle|^2\\
&& +\frac{1}{4}\langle (|\mathbb{S}|^2 +|\mathbb{S}^*|^2)x,x\rangle|\langle \mathbb{S}^2x,x\rangle|,~~\mbox{using Lemma \ref{positive}}\\
&=& \frac{1+\alpha}{8} \left \langle \left[\begin{array}{cc}
	|C|^4 +|B^*|^4& 0\\
	0& |B|^4+|C^*|^4
\end{array}\right]x,x \right \rangle\\	
&&+\frac{1-\alpha}{4}	\left | \left \langle \left[\begin{array}{cc}
BC & 0\\
0& CB
\end{array}\right]x,x    \right \rangle \right|^2\\
&& +\frac{1}{4} \left \langle \left[\begin{array}{cc}
	|C|^2 +|B^*|^2& 0\\
	0& |B|^2+|C^*|^2
\end{array}\right]x,x \right \rangle	\left | \left \langle \left[\begin{array}{cc}
BC & 0\\
0& CB
\end{array}\right]x,x    \right \rangle \right|\\
& \leq & \frac{1+\alpha}{8} w\left ( \left[\begin{array}{cc}
	|C|^4 +|B^*|^4& 0\\
	0& |B|^4+|C^*|^4
\end{array}\right] \right )+\frac{1-\alpha}{4}	w^2 \left ( \left[\begin{array}{cc}
	BC & 0\\
	0& CB
\end{array}\right]\right )\\
&& +\frac{1}{4} w \left ( \left[\begin{array}{cc}
	|C|^2 +|B^*|^2& 0\\
	0& |B|^2+|C^*|^2
\end{array}\right] \right )	w \left ( \left[\begin{array}{cc}
	BC & 0\\
	0& CB
\end{array}\right]    \right )\\
&=& \frac{1+\alpha}{8} \max \left\{\|  |B|^4+|C^*|^4\|, \|  |B^*|^4+|C|^4 \|  \right \}\\
&& +\frac{1-\alpha}{4} \max \left\{ w^2(BC),w^2(CB) \right\}\\
&&+ \frac{1}{4} \max \left\{\|  |B|^2+|C^*|^2\|, \|  |B^*|^2+|C|^2 \|  \right \} \times \max \left\{ w(BC),w(CB) \right\}.
\end{eqnarray*} 
Taking  supremum over all $x \in \mathcal{H}, \|x\|=1$, we get 
\begin{eqnarray*}
	&& 	w^4 \left(\left[\begin{array}{cc}
		0 & B\\
		C& 0
	\end{array}\right]\right)\\ &&\leq\frac{1+\alpha}{8} \max \left\{\|  |B|^4+|C^*|^4\|, \|  |B^*|^4+|C|^4 \|  \right \}\\ &&+\frac{1-\alpha}{4} \max \left\{ w^2(BC),w^2(CB) \right\}\\
	&&+ \frac{1}{4} \max \left\{\|  |B|^2+|C^*|^2\|, \|  |B^*|^2+|C|^2 \|  \right \} \times \max \left\{ w(BC),w(CB) \right\}.
\end{eqnarray*}
\end{proof}

In particular, considering $B=C$ in Theorem \ref{th-7}, we get the following corollary.
\begin{cor}\label{cor1}
If $B\in \mathcal{B}(\mathcal{H})$, then for $0\leq \alpha \leq 1$,
\begin{eqnarray*}
w^4(B) \leq\frac{1+\alpha}{8}  \left \|  |B|^4+|B^*|^4 \right \| +\frac{1-\alpha}{4} w^2(B^2)
+ \frac{1}{4}\left  \|  |B|^2+|B^*|^2 \right \|w(B^2). 
\end{eqnarray*}
\end{cor}

\begin{remark}
	For every $0\leq \alpha \leq 1,$ we have
	\begin{eqnarray*}
		w^4(B) &\leq& \frac{1+\alpha}{8}  \|  |B|^4+|B^*|^4\| +\frac{1-\alpha}{4} w^2(B^2)
		+ \frac{1}{4} \|  |B|^2+|B^*|^2\|w(B^2)\\
&\leq &	\frac{1+\alpha}{8}  \|  |B|^4+|B^*|^4\| +\frac{1-\alpha}{4} \|B^2\|^2
	+ \frac{1}{4} \|  |B|^2+|B^*|^2\|\|B^2\|\\	
	&\leq &	\frac{1+\alpha}{8}  \|  |B|^4+|B^*|^4\| +\frac{1-\alpha}{4} \left \|\frac{|B|^2+|B^*|^2}{2} \right \|^2\\
	&& + \frac{1}{4} \|  |B|^2+|B^*|^2\| \left \|\frac{|B|^2+|B^*|^2}{2}\right \|, ~~\,\, \left \|B^2 \right \|\leq \frac{1}{2} \left \||B|^2+|B^*|^2\right \|\\	
		&= &	\frac{1+\alpha}{8}  \|  |B|^4+|B^*|^4\| +\frac{1-\alpha}{4} \left \|\left (\frac{|B|^2+|B^*|^2}{2} \right)^2 \right \|\\
		&&+ \frac{1}{2} \left \|\left (\frac{|B|^2+|B^*|^2}{2} \right)^2 \right \|\\
		&\leq &	\frac{1+\alpha}{8}  \|  |B|^4+|B^*|^4\| +\frac{1-\alpha}{8}\|  |B|^4+|B^*|^4\|\\
		&& + \frac{1}{4} \|  |B|^4+|B^*|^4\|,~~\,\,\,\left \|\left (\frac{|B|^2+|B^*|^2}{2} \right)^2 \right \|\leq \left \|\frac{|B|^4+|B^*|^4}{2} \right \| \\
		&=& \frac{1}{2} \left  \|  |B|^4+|B^*|^4 \right \|.	
	\end{eqnarray*}
Hence, Corollary \ref{cor1} refines the earlier related  inequality  $w^4(B) \leq \frac{1}{2} \left  \|  |B|^4+|B^*|^4 \right \|,$ (see \cite{EK07}, for $r=2$).
\end{remark}

We next obtain the following estimation for an upper bound of  the numerical radius of general  $2\times 2$ operator matrices, i.e.,  $w \left(\left[\begin{array}{cc}
A & B\\
C& D
\end{array}\right]\right)$.

\begin{theorem}\label{th-25}
	If $A,B,C,D\in \mathcal{B}(\mathcal{H})$, then for $0\leq \alpha \leq 1$
	\begin{eqnarray*}
		 	w^4 \left(\left[\begin{array}{cc}
			A & B\\
			C& D
		\end{array}\right]\right) &\leq&  8 \max \left \{w^4(A),w^4(D)  \right \}\\
	 &&+({1+\alpha}) \max \left\{\|  |B|^4+|C^*|^4\|, \|  |B^*|^4+|C|^4 \|  \right \}\\
	 && +2(1-\alpha) \max \left\{ w^2(BC),w^2(CB) \right\}\\
		&&+ 2 \max \left\{\|  |B|^2+|C^*|^2\|, \|  |B^*|^2+|C|^2 \|  \right \}\\
		&& \times \max \left\{ w(BC),w(CB) \right\}.
	\end{eqnarray*}		
	
\end{theorem}

\begin{proof}
	Let $x\in \mathcal{H} \oplus \mathcal{H}$ with $\|x\|=1.$ Now have by convexity of $f(t)=t^4$,
	\begin{eqnarray*}
\left | \left \langle \left[\begin{array}{cc}
	A & B\\
	C& D
\end{array}\right]x,x \right \rangle \right|^4 &\leq & \left (\left | \left \langle \left[\begin{array}{cc}
A & 0\\
0& D
\end{array}\right]x,x \right \rangle \right| +\left | \left \langle \left[\begin{array}{cc}
0 & B\\
C& 0
\end{array}\right]x,x \right \rangle \right| \right)^4 \\
&\leq & 8\left (\left | \left \langle \left[\begin{array}{cc}
	A & 0\\
	0& D
\end{array}\right]x,x \right \rangle \right|^4 +\left | \left \langle \left[\begin{array}{cc}
	0 & B\\
	C& 0
\end{array}\right]x,x \right \rangle \right|^4 \right) \\
&\leq &8 w^4 \left (  \left[\begin{array}{cc}
	A & 0\\
	0& D
\end{array}\right]\right) +8 w^4\left( \left[\begin{array}{cc}
	0 & B\\
	C& 0
\end{array}\right] \right)\\
&=& 8 \max \left \{w^4(A),w^4(D)  \right \}+ 8 w^4\left( \left[\begin{array}{cc}
	0 & B\\
	C& 0
\end{array}\right] \right). 
	\end{eqnarray*}
Taking  supremum over all $x \in \mathcal{H}, \|x\|=1$ we have,
\begin{eqnarray*}
	w^4 \left(\left[\begin{array}{cc}
		A & B\\
		C& D
	\end{array}\right]\right) &\leq & 8 \max \left \{w^4(A),w^4(D)  \right \}+ 8 w^4\left( \left[\begin{array}{cc}
	0 & B\\
	C& 0
\end{array}\right] \right). 
\end{eqnarray*}
Therefore, by using Theorem \ref{th-7}, we get 
\begin{eqnarray*}
	&& 	w^4 \left(\left[\begin{array}{cc}
		A & B\\
		C& D
	\end{array}\right]\right)\\ &&\leq 8 \max \left \{w^4(A),w^4(D)  \right \}\\
	&&+({1+\alpha}) \max \left\{\|  |B|^4+|C^*|^4\|, \|  |B^*|^4+|C|^4 \|  \right \}\\
	&& +2(1-\alpha) \max \left\{ w^2(BC),w^2(CB) \right\}\\
	&&+ 2 \max \left\{\|  |B|^2+|C^*|^2\|, \|  |B^*|^2+|C|^2 \|  \right \} \times \max \left\{ w(BC),w(CB) \right\}.
\end{eqnarray*}	

\end{proof}

\begin{remark}
	It follows from \cite{D09} that
	$w(CB)\leq \frac{1}{2}\|  |B|^2+|C^*|^2\|~~\text{and}~~w(BC)\leq \frac{1}{2}\|  |B^*|^2+|C|^2\|.$	
	Therefore, clearly it follows that the inequality obtained in Theorem \ref{th-25} is stronger than the recently obtained inequality \cite[Th. 3.1]{BK21}, that is, 
\begin{eqnarray*}
	w^4 \left(\left[\begin{array}{cc}
		A & B\\
		C& D
	\end{array}\right]\right) &\leq&  8 \max \left \{w^4(A),w^4(D)  \right \}\\
	&&+({1+\alpha}) \max \left\{\|  |B|^4+|C^*|^4\|, \|  |B^*|^4+|C|^4 \|  \right \}\\
	&&+ (3-\alpha)\max \left\{\|  |B|^2+|C^*|^2\|, \|  |B^*|^2+|C|^2 \|  \right \} \\
	&& \times \max \left\{ w(BC),w(CB) \right\}.
\end{eqnarray*}

\end{remark}

\section{\textbf{Application}}
\noindent As application of results  obtained bounds in Section 2, we develope some norm inequalities for  sums and differences of self-adjoint operators. Note that if $B,C\in \mathcal{B}(\mathcal{H})$ are positive then  $w \left(\left[\begin{array}{cc}
0 & B\\
C& 0
\end{array}\right]\right)=\frac{\|B+C\|}{2}$, (see \cite[Cor. 3]{AK15}). Now we prove the following proposition, though it is known the proff given here is simple and different.

\begin{prop}\label{prop3.1}
	If $B,C\in \mathcal{B}(\mathcal{H})$ are positive, then 
	\begin{eqnarray*}
		&&(i)~~ \|B-C\|\leq \|B+C\|,\\
		&& (ii)~~ \max \Big \{\|B\|, \|C\| \Big\}\leq \frac{\|B+C\|}{2}+ \frac{\|B-C\|}{2}.
	\end{eqnarray*}
\end{prop}

\begin{proof}
From the first inequality in \ref{th-4}, we have
	\begin{eqnarray*}
\frac{\|B+C\|^2}{4} &\geq & \frac{1}{4\sqrt{2}}\left[\|B+C\|^4+\|B-C\|^4 \right]^{\frac{1}{2}}.
	\end{eqnarray*}
This implies that 	$\|B-C\|\leq \|B+C\|$, i.e, (i).  Now from Theorem \ref{th-5} we have,
\[\frac{\|B+C\|}{2} \geq \frac{1}{2} \max \{\|B\|, \|C\|\}+ \frac{1}{4}|\|B+C\|-\|B-C\| |.\]
Therefore, using (i) we have,
\[\frac{\|B+C\|}{2} \geq \frac{1}{2} \max \{\|B\|, \|C\|\}+ \frac{1}{4}(\|B+C\|-\|B-C\| ).\]
This completes the proof of (ii).
\end{proof}
 
Next we prove the following.

\begin{theorem}\label{prop3.3}
	
	Let $B,C\in \mathcal{B}(\mathcal{H})$ be self-adjoint. Then,	
	\[\max \Big\{\|B+C\|^2, \|B-C\|^2\Big\} \leq \left \|B^2+C^2 \right\|+2w(|B||C|).\]
	
\end{theorem}

\begin{proof}
We have $$w\left(\left[\begin{array}{cc}
0 & B\\
C& 0
\end{array}\right]\right)\geq \left\|\Re \left(\left[\begin{array}{cc}
0 & B\\
C& 0
\end{array}\right]\right) \right\|$$ 
and 
$$ w\left(\left[\begin{array}{cc}
0 & B\\
C& 0
\end{array}\right]\right)\geq \left\|\Im \left( \left[\begin{array}{cc}
0 & B\\
C& 0
\end{array}\right]\right) \right\|$$
so that 
	  $$  w\left(\left[\begin{array}{cc}
	0 & B\\
	C& 0
	\end{array}\right]\right) \geq \left \|\frac{B+C}{2}\right \|$$ and $$  w\left(\left[\begin{array}{cc}
	0 & B\\
	C& 0
	\end{array}\right]\right)\geq \left \|\frac{B-C}{2}\right \|$$ respectively.
  Therefore,
	\begin{eqnarray*}
		 \frac{1}{4} \max \left \{  \left \|{B+C}\right \|^2,  \left \|{B-C}\right \|^2 \right \} &\leq& w^2\left(\left[\begin{array}{cc}
		 	0 & B\\
		 	C& 0
		 \end{array}\right]\right).
	 \end{eqnarray*}
	Hence, using Theorem \ref{th-2} we get
	\[\max \Big\{\|B+C\|^2, \|B-C\|^2\Big\} \leq \left \|B^2+C^2 \right\|+2w(|B||C|).\]
	 This completes the proof.
	
\end{proof}

\begin{remark} \label{rem}
(i)	It follows from the triangle inequality of the numerical radius that if $B,C\in \mathcal{B}(\mathcal{H})$ are self-adjoint, then	
\[\max \Big\{\|B+C\|^2, \|B-C\|^2\Big\} \leq \left \|B^2+C^2 \right\|+2w(BC).\]
(ii) Clearly, if $B,C$ are positive then the inequalities in Theorem \ref{prop3.3} and Remark \ref{rem}(i) are same.
In \cite{K2002}, Kittaneh proved that if $B,C\in \mathcal{B}(\mathcal{H})$ are positive, then
	\begin{eqnarray*}
\|B+C\|&\leq &\frac{1}{2}\left[ \|B\|+\|C\|+\sqrt{(\|B\|-\|C\|)^2+4\left \|B^{1/2}C^{1/2} \right\|^2}\right].
	\end{eqnarray*}
 In the example given below, we note that the bound obtained in Theorem \ref{prop3.3} (for positive operators) is better than that in  \cite{K2002}. Consider $B=\left[\begin{array}{cc}
4 & 0\\
0& 0
\end{array}\right]$ and $C=\left[\begin{array}{cc}
1& 0\\
0& 2
\end{array}\right].$ Then, Theorem \ref{prop3.3} gives $\|B+C\|\leq 5$,  whereas \cite{K2002} gives $ $ $\|B+C\|\leq 3+\sqrt{5}$.
\end{remark}

\begin{remark}
	Let $B,C\in \mathcal{B}(\mathcal{H})$ be self-adjoint. 
	It follows from Theorem \ref{prop3.3} and Remark \ref{rem}(i) that if  $\|B+ C\|=\|B\|+\|C\|,$ then
	\begin{eqnarray*}
		&&(i)~~ \left \|B^2+C^2 \right\|=\|B\|^2+\|C\|^2,\\
		&& (ii)~~ w(|B||C|)=\|BC\|=\|B\|\|C\|=w(BC).
	\end{eqnarray*}
The converse of the above result does not hold, in general. As for example consider $B=\left[\begin{array}{cc}
1 & 0\\
0& 1
\end{array}\right]$ and $C=\left[\begin{array}{cc}
-1& 0\\
0& -1 
\end{array}\right].$ Then we see that $\left \|B^2+C^2 \right\|=\|B\|^2+\|C\|^2=2$ and $w(|B||C|)=\|BC\|=\|B\|\|C\|=w(BC)=1$, but $0=\|B+C\| \neq \|B\|+\|C\|=2.$
We note that (see \cite{K2002}) when $B,C$ are positive, then  $\|B+C\|=\|B\|+\|C\|$ if and only if $\|BC\|=\|B\|\|C\|.$
\end{remark}

\bibliographystyle{amsplain}

\end{document}